\documentclass{amsart}
\usepackage[utf8]{inputenc}
\usepackage{amsmath}
\usepackage{amsthm}
\usepackage{amsfonts}
\usepackage{amssymb}
\usepackage[all]{xy}
\usepackage{indentfirst}

\newtheorem{thm}{Theorem}[section]

\newtheorem{theorem}[thm]{Theorem}

\newtheorem{lemma}[thm]{Lemma}
\newtheorem*{theorem*}{Theorem}

\theoremstyle{definition}

\newtheorem{remark}[thm]{Remark}

\newcommand{\N}{\mathbb{N}} 
\newcommand{\Z}{\mathbb{Z}} 
\newcommand{\C}{\mathbb{C}} 
\newcommand{\Q}{\mathbb{Q}} 
\usepackage{verbatim} 
\usepackage{color}

\newcommand{\cs}{\mathrm{C}^*}

\title{A torsion-free algebraically C*-unique group}
\author{Eduardo Scarparo}
\thanks{This work was carried out during the tenure of an ERCIM ‘Alain Bensoussan’ Fellowship Programme.}
\address{Department of Mathematical Sciences, NTNU, NO-7491 Trondheim, Norway}
\email{eduardo.scarparo@ntnu.no}

\begin{document}

\begin{abstract}
Let $p$ and $q$ be multiplicatively independent integers. We show that the complex group ring of $\mathbb{Z}[\frac{1}{pq}]\rtimes\Z^2$ admits a unique $\mathrm{C}^*$-norm. The proof uses a characterization, due to Furstenberg, of closed $\times p-$ and $\times q-$invariant subsets of $\mathbb{T}$. 
\end{abstract}

\maketitle

\section{Introduction}
 
Given a group $G$, there are in general many $\mathrm{C}^*$-norms on its complex group ring $\C [G]$. For example, if $G=\Z$, identifying $\C[\Z]$ with complex polynomials $p(z)=\sum\alpha_nz^n$ defined on $\mathbb{T}$, we have that any infinite closed subset $F\subset \mathbb{T}$ gives rise to a distinct $\mathrm{C}^*$-norm $\|\cdot\|$ on $\C[\Z]$, defined by $\|p\|:=\sup_{z\in F}|p(z)|$, for $p\in\C[\Z]$ (this was noted in \cite[Chapter 19]{zbMATH06801030} and \cite{zbMATH06856863}). . 

Following \cite{zbMATH07050829}, we say that a group $G$ is \emph{algebraically $\mathrm{C}^*$-unique} if $\C [G]$ admits a unique $\mathrm{C}^*$-norm. 

Clearly, any algebraically $\mathrm{C}^*$-unique group must be amenable. In \cite{zbMATH06856863}, Grigorchuk, Musat and Rørdam proved that any locally finite group is algebraically $\mathrm{C}^*$-unique and asked whether the converse holds. Furthermore, Alekseev and Kyed obtained in \cite{zbMATH07050829} large classes of amenable groups which are not algebraically $\mathrm{C}^*$-unique. 

In \cite{zbMATH07118831}, Caspers and Skalski studied the question of $\mathrm{C}^*$-uniqueness in the context of discrete quantum groups, and showed that there exists a $\mathrm{C}^*$-unique discrete quantum group which is not locally finite.

As told by Alekseev in his report for the 2019 workshop ``$\cs$-algebras" at Oberwolfach \cite{Alekseevoberwolfach}, Ozawa pointed out, during the meeting, that the lamplighter group $(\bigoplus_\Z\Z_2)\rtimes\Z$ is algebraically $\mathrm{C}^*$-unique and not locally finite, thus answering in the negative the question in \cite{zbMATH06856863}. Alekseev asked then in \cite{Alekseevoberwolfach} whether there exists a \emph{torsion-free} group which is algebraically $\mathrm{C}^*$-unique.

In this paper, by adapting Ozawa's argument, we provide an example of a torsion-free algebraically $\mathrm{C}^*$-unique group $G$. The proof that $G$ is algebraically $\mathrm{C}^*$-unique uses a certain result by Furstenberg (\cite{zbMATH03236196}) on diophantine approximation.

\section{A torsion-free algebraically $\mathrm{C}^*$-unique group}

Fix $p$ and $q$ integer such that $p,q\geq 2$ and such that there are no $r,s\in\N$ such that $p^r=q^s$ (this means that $p$ and $q$ are \emph{multiplicatively independent}). 

Let $\Z[\frac{1}{pq}]$ be the additive group $\{\frac{a}{(pq)^n}\in\Q:a\in\Z,n\in\N\}$, and $\alpha$ the action of $\Z^2$ on $\Z[\frac{1}{pq}]$ given by $\alpha_{(n,m)}(x):=p^nq^mx$, for $n,m\in\Z$ and $x\in \Z[\frac{1}{pq}]$. We will show that the trosion-free group $\mathbb{Z}[\frac{1}{pq}]\rtimes\Z^2$ is algebraically $\cs$-unique. By \cite[Lemma 2.2]{zbMATH07050829}, this is equivalent to showing that, given an ideal $I\unlhd \mathrm{C}^*(\Z[\frac{1}{pq}]\rtimes\Z^2)$ such that $I\cap\C[\Z[\frac{1}{pq}]\rtimes\Z^2]=\{0\}$, we have that $I=\{0\}$.

\begin{remark}

In \cite{zbMATH06761943}, Huang and Wu  studied unitary representations of $\Z[\frac{1}{pq}]\rtimes\Z^2$ in connection with Furstenberg's conjecture on $\times p-$ and $\times q-$invariant probability measures on $\mathbb{T}$.

\end{remark}

Notice that, because $p$ and $q$ are multiplicatively independent, we have that, if $n,m\in \Z$ are not both zero, then $p^nq^m\neq 1$. This implies that the action of $\Z^2$ on $\Z[\frac{1}{pq}]$ is faithful. 

Recall that an action of a group $G$ on a locally compact Hausdorff space $X$ is
said to be \emph{topologically free} if, for each $g\in G\setminus\{e\}$, the set of points of $X$ fixed by
$g$ has empty interior.

The following lemma follows easily from \cite[Lemma 2.1]{zbMATH06465649}. For the sake of completeness, we include a proof.

\begin{lemma}
Let $A$ be a torsion-free discrete abelian group and $\beta\colon G\curvearrowright A$ a faithful action. Then $\widehat{\beta}\colon G\curvearrowright \widehat{A}$ is topologically free.
\end{lemma}
\begin{proof}

Take $g\in G$ such that the set $F_g$ of points of $\widehat{A}$ fixed by $\widehat{\beta}_g$ has non-empty interior; we will show that $g=e$.

Since $F_g$ is a subgroup of $\widehat{A}$, the fact that $F_g$ has non-empty interior implies that $F_g$ is open. Moreover, since $\widehat{\beta}_g$ is continuous, we also have that $F_g$ is closed.

From the fact that $A$ is torsion-free, we obtain that $\widehat{A}$ is connected, hence $F_g=\widehat{A}$. Since $\beta$ is a faithful action, we conclude that $g=e$.

\end{proof}

From the lemma above, we obtain that the action $\widehat{\alpha}\colon\Z^2\curvearrowright\widehat{\Z[\frac{1}{pq}]}$ is topologically free. Hence, given a non-zero ideal $I\unlhd \mathrm{C}^*(\Z[\frac{1}{pq}]\rtimes \Z^2)\simeq C(\widehat{\Z[\frac{1}{pq}]})\rtimes\Z^2,$
 we have that $I\cap \mathrm{C}^*(\Z[\frac{1}{pq}])\neq\{0\}$ (for a proof of this general fact about topologically free actions, see, for instance, \cite[Theorem 29.5]{zbMATH06801030}).

Let 
\begin{align*}
\varphi\colon \mathbb{T}&\to \mathbb{T}\\
z&\mapsto z^{pq}
\end{align*}

and $X:=\varprojlim(\mathbb{T},\varphi)=\{(x_n)_{n\in\N}\in\prod_{n\in\N}\mathbb{T}:\forall n\in\N,x_n=\varphi(x_{n+1})\}$.

\begin{lemma}\label{tec}
There is an isomorphism $\tilde{\psi}\colon\cs(\Z[\frac{1}{pq}])\to C(X)$ such that $$\tilde{\psi}(\delta_{\frac{a}{(pq)^m}})(x)=(x_m)^a,$$ for $a\in\Z$, $m\in\N$ and $x=(x_n)_{n\in\N}\in X$.
\end{lemma}
\begin{proof}
Let $$\mathrm{Ev}\colon \cs\left(\Z\left[\frac{1}{pq}\right]\right)\to C\left(\widehat{\Z\left[\frac{1}{pq}\right]}\right)$$ be the isomorphism given by point-evaluation, i.e., given $u\in \Z[\frac{1}{pq}]$ and $\tau\in\widehat{\Z[\frac{1}{pq}]}$, we have that $\mathrm{Ev}(\delta_u)(\tau)=\tau(u)$.

Let $H\colon \widehat{\Z[\frac{1}{pq}]}\to X$ be the continuous map given by $H(\tau):=(\tau(\frac{1}{(pq)^n}))_{n\in\N}$, for $\tau\in\widehat{\Z[\frac{1}{pq}]}$. Also let $\tilde{H}\colon C(X)\to C(\widehat{\Z[\frac{1}{pq}]})$ be the homomorphism induced by $H$.

Let $\psi\colon \Z[\frac{1}{pq}]\to C(X)$ be given by $\psi(\frac{a}{(pq)^m})(x)=(x_m)^a$, for $a\in\Z$, $m\in\N$ and $x=(x_n)_{n\in\N}\in X$. It is straightforward to check that $\psi$ is a well-defined unitary representation of $\Z[\frac{1}{pq}]$. Let $\tilde{\psi}\colon \cs(\Z[\frac{1}{pq}])\to C(X)$ be the canonical extension of $\psi$. An application of the Stone-Weierstrass theorem shows that $\tilde{\psi}$ is surjective.

Furthermore, it can be readily checked that $\tilde{H}\circ\tilde{\psi}=\mathrm{Ev}$. Since we know that $\mathrm{Ev}$ is an isomorphism, we conclude that $\tilde{\psi}$ is also an isomorphism.
\end{proof}

We denote by $\tilde{\alpha}$ the action of $\Z^2$ on $\cs(\Z[\frac{1}{pq}])$ induced by $\alpha$. There is an action $\beta\colon\Z^2\curvearrowright X$ such that, for $f\in C(X)$ and $(r,s)\in\Z^2$, we have that $\tilde{\psi}\circ\tilde{\alpha}_{(r,s)}\circ\tilde{\psi}^{-1}(f)=f\circ\beta_{(r,s)}^{-1}$. One can readily check that, for $x\in X$ and $r,s\in\Z$ non-negative integers, it holds that $\beta_{(r,s)}^{-1}(x)=x^{p^rq^s}$. Furthermore, $\beta_{(1,1)}$ is the left shift on $X$.

We will need the following result of Furstenberg (\cite{zbMATH03236196}), whose precise formulation we take from \cite[Theorem 1.2]{zbMATH00672257}:

\begin{theorem}\label{furs}
If $B\subset \mathbb{T}$ is an infinite closed set which is $\times p-$ and $\times q-$invariant, then $B=\mathbb{T}$.
\end{theorem}

\begin{theorem}
The group $\Z[\frac{1}{pq}]\rtimes\Z^2$ is algebraically $\cs$-unique.
\end{theorem}

\begin{proof}

Let $I\unlhd C(X)\rtimes\Z^2$ be an ideal and suppose that, under the identification given in Lemma \ref{tec}, we have that $I\cap\C[\Z[\frac{1}{pq}]]=\{0\}$. We will show that $I\cap C(X)=\{0\}$, and therefore $I=\{0\}$.

Notice that $C(X)\cap I$ is a $\Z^2$-invariant ideal of $C(X)$, hence there is $F\subset X$ a $\Z^2$-invariant closed set such that $C(X)\cap I=C_0(F^c)$. 

For $n\in\N$, let $\pi_n\colon X\to \mathbb{T}$ be the canonical projection. Let $B:=\pi_1(F)$. The fact that $F$ is $\Z^2$-invariant implies that $B=\pi_n(F)$ for $n\in\N$ and that $$B=\{z^p:z\in B\}=\{z^q:z\in B\}.$$

Since $I\cap \C[\Z[\frac{1}{pq}]]=\{0\}$, we have that $B$ contains infinitely many points, for otherwise there would be a non-zero polynomial vanishing on $B$. 

Using Theorem \ref{furs}, we conclude that $B=\mathbb{T}$, hence $F=X$ and $I=\{0\}$.
\end{proof}
\bibliographystyle{acm}
\bibliography{bibliografia}

\begin{thebibliography}{1}

\bibitem{Alekseevoberwolfach}
{\sc Alekseev, V.}
\newblock {(Non)-uniqueness of \(\mathrm{C}^\ast\)-norms on group rings of
  amenable groups}.
\newblock {\em Oberwolfach Rep.}, 37 (2019).

\bibitem{zbMATH07050829}
{\sc {Alekseev}, V., and {Kyed}, D.}
\newblock {Uniqueness questions for \(C^\ast\)-norms on group rings.}
\newblock {\em {Pac. J. Math.} 298}, 2 (2019), 257--266.

\bibitem{zbMATH00672257}
{\sc {Boshernitzan}, M.~D.}
\newblock {Elementary proof of Furstenberg's diophantine result.}
\newblock {\em {Proc. Am. Math. Soc.} 122}, 1 (1994), 67--70.

\bibitem{zbMATH07118831}
{\sc {Caspers}, M., and {Skalski}, A.}
\newblock {On \(\mathrm{C}^\ast\)-completions of discrete quantum group rings.}
\newblock {\em {Bull. Lond. Math. Soc.} 51}, 4 (2019), 691--704.

\bibitem{zbMATH06465649}
{\sc {Eckhardt}, C.}
\newblock {A note on strongly quasidiagonal groups.}
\newblock {\em {J. Oper. Theory} 73}, 2 (2015), 417--424.

\bibitem{zbMATH06801030}
{\sc {Exel}, R.}
\newblock {\em {Partial dynamical systems, Fell bundles and applications.}},
  vol.~224.
\newblock Providence, RI: American Mathematical Society (AMS), 2017.

\bibitem{zbMATH03236196}
{\sc {Furstenberg}, H.}
\newblock {Disjointness in ergodic theory, minimal sets, and a problem in
  Diophantine approximation.}
\newblock {\em {Math. Syst. Theory} 1\/} (1967), 1--49.

\bibitem{zbMATH06856863}
{\sc {Grigorchuk}, R., {Musat}, M., and {R{\o}rdam}, M.}
\newblock {Just-infinite \(C^\ast\)-algebras.}
\newblock {\em {Comment. Math. Helv.} 93}, 1 (2018), 157--201.

\bibitem{zbMATH06761943}
{\sc {Huang}, H., and {Wu}, J.}
\newblock {Ergodic invariant states and irreducible representations of crossed
  product \(C^*\)-algebras.}
\newblock {\em {J. Oper. Theory} 78}, 1 (2017), 159--172.

\end{thebibliography}
\end{document}